\documentclass[a4paper]{amsart}
 
\usepackage{amsthm}
\usepackage{amsmath}
\usepackage{amssymb}
\usepackage{latexsym}
\usepackage{graphicx}
\usepackage[utf8]{inputenc}
\usepackage{epsfig} 
\usepackage{pgf}
\usepackage{tikz}
\usepackage{float}
\usepackage{dsfont}
\usepackage{hyperref}

\newtheorem*{theorem*}{Theorem}

\renewcommand{\thetheoremName}

\newtheorem{theorem}{Theorem}[section]
\newtheorem{Theorem}{Theorem}

\newtheorem{proposition}[theorem]{Proposition}
\newtheorem{corollary}[theorem]{Corollary}

\newtheorem{definition}[theorem]{Definition}
\newtheorem{example}[theorem]{Example}

\newtheorem{remark}[theorem]{Remark}

\numberwithin{equation}{section}

\begin{document}

\title[First Eigenvalue of the Laplacian  of a Geodesic Ball]{First Eigenvalue of the Laplacian of a Geodesic Ball and  Area-Based Symmetrization of its Metric Tensor}

\author[V. Gimeno]{Vicent Gimeno*}
\address{Departament de Matem\`{a}tiques- IMAC,
Universitat Jaume I, Castell\'{o}, Spain.}
\email{gimenov@uji.es}
\author[E. Sarrion-Pedralva]{Erik Sarrion-Pedralva**}
\address{Departament de Matem\`{a}tiques-IMAC,
Universitat Jaume I, Castell\'o, Spain. }
\email{eriksarrionpedralva@gmail.com}
\thanks{* Work partially supported by the Research Program of University Jaume I Project UJI-B2018-35, and DGI -MINECO grant (FEDER) MTM2017-84851-C2-2-P}
\thanks{** Work partially supported by the Research Program of University Jaume I Project UJI-B2018-35, and DGI -MINECO grant (FEDER) MTM2017-84851-C2-2-P, and grant for predoctoral research GVA-ESF ACIF/2019/096.}

\begin{abstract}
Given a Riemmanian manifold,  we provide a new method to compute a sharp upper bound for the first eigenvalue of the Laplacian for the Dirichlet problem on a geodesic ball of radius less than the injectivity radius of the manifold. This upper bound is obtained by  transforming  the metric tensor into a rotationally symmetric metric tensor that preserves the area of the geodesic spheres.   The provided upper bound can be computed using only the area function of the geodesic spheres contained in the geodesic ball and it is sharp in the sense that the first eigenvalue of geodesic ball coincides with our upper bound if and only if the mean curvature pointed inward of each geodesic sphere is a radial function.
\end{abstract}

\maketitle

\section{Introduction}\label{Intr} 
Let $(M,g)$ be a $n$-dimensional Riemannian manifold, and let $\Omega\subset M$ be a precompact domain with smooth boundary $\partial \Omega$. The first eigenvalue $\lambda_{1,g}(\Omega)$ of the Laplacian for the Dirichlet problem on $\Omega$ is the smallest $\lambda\in \mathbb{R}$ such that there exists a non trivial associated eigenfunction $\phi$ satisfying
$$
\left\{
	\begin{array}{rcll}
		\Delta_g\phi+\lambda\phi&=0,&\text{ on }&\Omega,\\
		\phi &=0,&\text{ on }&\partial \Omega,
	\end{array}
\right.
$$
where $\Delta_g$ is the Laplacian operator with respect to the metric tensor $g$, \emph{i.e.}, 
$$
    \Delta_g=\frac{1}{\sqrt{\det g}}\sum_{i,j=1}^n\frac{\partial}{\partial x^i}\left(\sqrt{\det g\,}g^{ij}\frac{\partial }{\partial x^j}\right)
$$
in local coordinates $(x^1,\cdots,x^n)$. Upper and lower bounds for the first eigenvalue have been widely studied in terms of geometric invariants of the domain $\Omega$.

In \cite{Cheng2} and \cite{Cheng1} Cheng obtained upper and lower bounds for the first eigenvalue of the Laplacian for the Dirichlet {problem on a geodesic ball as follows:}  {I}f the Ricci curvatures ${\rm Ric}_g$ of $M$ are {bounded from below} by the Ricci curvatures of a  $n$-dimensional simply connected real space form $(\mathbb{M}_\kappa^n,g_\kappa)$ of constant sectional curvature $\kappa$, \emph{i.e.}, 
$${\rm Ric}_g\geq(n-1)\kappa,$$ 
then the first eigenvalue $\lambda_{1,g}({\rm B}_R(p))$ of the Laplacian for the Dirichlet problem of a geodesic ball ${\rm B}_R(p)$ of radius $R$ centered at $p\in M$ is bounded from above by
\begin{equation}\label{eqCheng1}
	\lambda_{1,g}({\rm B}_R(p))\leq\lambda_1(R,\kappa),    
\end{equation} 
where $\lambda_1(R,\kappa)$ is the first Dirichlet eigenvalue of  a geodesic ball of radius $R$ in $\mathbb{M}_\kappa^n$.
Conversely, if the sectional curvatures of $M$ are bounded {from above} by the sectional curvature of  $(\mathbb{M}_\kappa^n,g_\kappa)$, \emph{i.e.}, 
$${\rm sec}_g\leq\kappa,$$
then the first eigenvalue $\lambda_{1,g}({\rm B}_R(p))$ of the Laplacian for the Dirichlet problem of a geodesic ball ${\rm B}_R(p)$ of radius $R<\min\{{\rm inj}_g(p), \pi/\sqrt{\kappa}\}$\footnote{where ${\rm inj}_g(p)$ is the injectivity radius of $p$ and $\pi/\sqrt{\kappa}$ is replaced by $+\infty$ if $\kappa\leq 0$.}  is bounded  {from below} by
\begin{equation}\label{eqCheng2}
    \lambda_1({\rm B}_R(p))\geq\lambda_1(R,\kappa).
\end{equation}
Moreover, inequalities  \eqref{eqCheng1} and \eqref{eqCheng2} are sharp because equality is attained in both inequalities if and only if ${\rm B}_R(p)$ is isometric to the geodesic ball of radius $R$ in $\mathbb{M}_\kappa^n$. 

In \cite{BM08}, Bessa and Montenegro, have obtained the same upper and lower bounds for the first eigenvalue of the geodesic ball $B_R(p)$ but {assuming that the} mean curvature of the geodesic spheres of $(M,g)$ centered at $p\in M$ {are bounded by the mean curvature of the geodesic spheres in  $\mathbb{M}_\kappa^n$}. In particular, if the mean curvature pointed inward ${\rm H}_{{\rm S}_t(p)}$ of the geodesic spheres ${\rm S}_t(p)$ of $M$ are {bounded from above  (resp. from below)} by the mean curvature  pointed inward ${\rm H}(t,\kappa)$ of the geodesic sphere with the same radius of $\mathbb{M}_\kappa^n$,  \emph{i.e.},
$${\rm H}_{{\rm S}_t(p)}\leq{\rm H}(t,\kappa)\quad \left(\vphantom{\dfrac{1}{1}}{\rm resp.} \quad{\rm H}_{{\rm S}_t(p)}\geq{\rm H}(t,\kappa)\right),$$
for all point in ${\rm S}_t(p)$ and for all $t\in(0,R)$, then 
\begin{equation}\label{eqpacelli}
	\lambda_{1,g}({\rm B}_R(p))\leq\lambda_1(R,\kappa)\quad \left(\vphantom{\dfrac{1}{1}}{\rm resp.}\quad \lambda_{1,g}({\rm B}_R(p))\geq\lambda_1(R,\kappa)\right).    
\end{equation}
In this case, {like in the case where is assumed  lower bounds for the Ricci or upper bounds for the sectional curvatures} , inequality \eqref{eqpacelli} is sharp. But now, instead of an isometry between geodesic balls, equality is attained in \eqref{eqpacelli} if and only if ${\rm H}_{{\rm S}_t(p)}={\rm H}(t,\kappa)$ for all $t\in(0,R)$.
Observe that {the conclusion of  an}  equality {between} the mean curvatures pointed inward of geodesic spheres is a weaker {than the conclusion of} an isometry between geodesic balls. { Indeed, in example 5.3 of \cite{BGJ19} is shown a $4$-dimensional geodesic ball non-isometric to the geodesic ball of $\mathbb{M}_\kappa^4$, but with ${\rm H}_{{\rm S}_t(p)}={\rm H}(t,\kappa)$ for all $t\in(0,R)$ .}

To bound the sectional or Ricci curvatures, or to  bound  the mean curvature of the geodesic spheres, imply to  control  the behavior of the isoperimetric quotients.  An alternative way  to obtain bounds for the first eigenvalue of the Laplacian makes use of the \emph{isoperimetric inequalities} and the so-called \emph{symmetrizations}.
One of the most classical symmetrizations is the \emph{Schwarz symmetrization} (see \cite{Ba, Po, Mc02, Ch3}). The Schwarz symmetrization of a compact open domain $D\subseteq M$ is the unique geodesic ball ${\rm B}_{R(D)}^\kappa$ of $\mathbb{M}_\kappa^n$ satisfying that ${\rm vol}\left(D\right)={\rm vol}\left({\rm B}_{R(D)}^\kappa\right)$. 

In \cite{Faber, Krahn}, Faber and Krahn showed that if for all open set $D\subseteq M$, consisting in a disjoint union of regular domains, the volume of the perimeter of $D$ is greater than the volume of the perimeter of its  Schwarz symmetrization ${\rm B}_{R(D)}^\kappa$, \emph{i.e.}, if
$$
{\rm vol}\left(\partial D\right)\geq {\rm vol}\left(\partial {\rm B}_{R(D)}^\kappa\right),
$$
then, for all precompact domain $\Omega\subset M$,
\begin{equation}\label{eq:faberkhran}
\lambda_1(\Omega)\geq\lambda_1\left(R(\Omega),\kappa\right).
\end{equation}
{Like in the results of  Cheng}, inequality \eqref{eq:faberkhran} is sharp in the sense that equality is attained if and only if $\Omega$ is isometric to ${\rm B}_{R(\Omega)}^\kappa$.

Instead of bounding the first eigenvalue assuming certain geometric hypothesis on the underlying Riemannian manifold, the first eigenvalue can be computed directly using the so-called Poisson hierarchy (see \cite{Mc02,Mc13,HMP15}). Let $(M,g)$ be a Riemannian manifold, the tower of moments of $\Omega\subset M$ is the family of functions $\displaystyle\{u_k\}_{k=0}^\infty$,  defined inductively as the following sequence of solutions to a hierarchy of boundary value problems in $\Omega\subset M$:

\noindent Let us define
\begin{equation*}
	u_0=1\quad \text{ on }\quad \Omega,
\end{equation*}
and for $k\geq 1$, 
$$
\left\{
	\begin{array}{rcll}
		\Delta_g u_k+k\,u_{k-1}&=0,&\text{ on }&\Omega,\\
		u_k &=0,&\text{ on }&\partial \Omega.
	\end{array}\right.
$$
 Moreover, the \emph{moment spectrum of $\Omega$} is defined as the family of  integrals
$$
\mathcal{A}_k(\Omega)=\int_\Omega u_k\,d{\rm V}_g.
$$
In \cite{McMy}, McDonald and Meyers proved that the moment spectrum of $\Omega$ can be used to compute the first eigenvalue of the Laplacian for the Dirichlet problem on $\Omega$ by
$$
\lambda_{1,g}(\Omega)= \sup \left\{\eta \geq 0 \,:\, \lim_{k \to \infty}
\sup
\left(\frac{\eta}{2}\right)^n\frac{\mathcal{A}_{k}(\Omega)}{\Gamma(n+1)}
<\infty\right\}.
$$
More recently, in \cite{BGJ19}, Bessa, Jorge and the first author have proved that
\begin{equation}\label{eq:lambda1BGJ}
	\lambda_{1,g}(\Omega)=\lim_{k\to\infty}(k+1)\dfrac{\lVert{u}_k\rVert_2}{\lVert{u}_{k+1}\rVert_2},
\end{equation}
where $\lVert\cdot\rVert_2$ denotes the $L_2$-norm on $\Omega$.  In the particular case when the domain is a geodesic ball $B_R^w$ of radius $R$ in a Riemannian model $\mathbb{M}_w^n$, see section \ref{secpropmet} for the definition of a Riemannian model, in \cite{HMP15} is proved that 
\begin{equation}\label{eq:HMP}
\lambda_{1,g_w}(B_R^w)=\lim_{k\to\infty}\frac{ku_{k-1}(0)}{u_k(0)}=\lim_{k\to\infty}\frac{k\mathcal{A}_{k-1}(B_R^w)}{\mathcal{A}_{k}(B_R^w)}.
\end{equation}

In this paper we prove that an upper bound for the first eigenvalue of the Laplacian for the Dirichlet problem of a geodesic ball $B_R(p)$ can be computed using only the \emph{area function of the geodesic spheres},
$${\rm A}_g(t):={\rm vol}_g({\rm S}_t(p)),$$
and the following family of functions $\{T_k\}_{k=0}^\infty$, constructed recursively from the area function, where 
$$
T_0(t)=1
$$
and for $k\geq 1$,
\begin{equation}\label{eq:defTk}
T_k(t)=\int_t^R\frac{\int_0^\sigma T_{k-1}(s){\rm A}_g(s)\,ds}{{\rm A}_g(\sigma)}\,d\sigma.
\end{equation}
Our upper bound for the first eigenvalue of ${\rm B}_R(p)$, obtained using the above family of functions $\{T_k\}$, is stated in the following
\begin{Theorem}\label{MainThm}
	Let $(M,g)$ be a Riemannian manifold, let $p\in M$ be a point of $M$ with injectivity radius ${\rm inj}_{g}(p)$, and let ${\rm B}_R(p)$ be the geodesic ball of radius $R$ centered at $p$. Suppose that $R<{\rm inj}_{g}(p)$, then the first eigenvalue $\lambda_{1,g}({\rm B}_R(p))$ of the Laplacian for the Dirichlet problem on ${\rm B}_R(p)$ is bounded by
	\begin{equation}\label{mainineq}
		\lambda_{1,g}({\rm B}_R(p))\leq \lim_{k\to \infty}\left(\frac{\int_0^RT_k^2(t){\rm A}_g(t)\,dt}{\int_0^RT_{k+1}^2(t){\rm A}_g(t)\,dt}\right)^{1/2}.
	\end{equation}
	Furthermore, equality is attained in \eqref{mainineq} if and only if, for any $t\in(0,R)$, the mean curvature pointed inward ${\rm H}_{{\rm S}_t(p)}$ of the geodesic sphere ${\rm S}_t(p)$ of radius $t$ centered at $p$ is a radial function. Namely, equality in \eqref{mainineq} is attained if and only if there exists a smooth function $h(t)$ such that
	$$
	{\rm H}_{{\rm S}_t(p)}=h(t) \quad {\rm for\; any}\quad 0<t<R.
	$$
\end{Theorem}

In this work instead of use the Schwarz symmetrization {(as  Faber and Krahn in \cite{Faber,Krahn})}, we will construct a rotationally symmetric metric tensor $\widetilde{g}$ for {the} geodesic ball ${\rm B}_R(p)$ with respect to the metric tensor $g$, such that the volume of the geodesic spheres of $({\rm B}_R(p),\widetilde{g})$ will coincide with the volume of the geodesic spheres with the same radius of $({\rm B}_R(p),g)$. Namely,
$$
{\rm A}_g(t)={\rm A}_{\widetilde{g}}(t).
$$
This rotationally symmetric metric tensor allows us to find a comparison for the first eigenvalue of the Laplacian for the Dirichlet problem but without using an isoperimetric inequality as hypothesis as in Faber and Krahn Theorem. 

The family of functions \eqref{eq:defTk} are related with the moment functions of the Poisson hierarchy $\{\widetilde{u}_k\}$ associated to $\widetilde{g}$. In fact, in \cite{BGJ19}, is proved that since $\widetilde{g}$ is a rotationally symmetric metric tensor
$$
{\widetilde{u}_k(q)=k! T_k({\rm dist}_{\widetilde g}(p,q))},\quad \Vert \widetilde{u}_k\Vert_2=k!\left(\int_0^RT^2_k(t) A_g(t)\,dt\right)^{1/2}.
$$
Hence by \eqref{eq:lambda1BGJ},
$$
\lambda_{1,\widetilde{g}}({\rm B}_R(p))=\lim_{k\to \infty}\left(\frac{\int_0^RT_k^2(t){\rm A}_g(t)\,dt}{\int_0^RT_{k+1}^2(t){\rm A}_g(t)\,dt}\right)^{{1/2}}.
$$
Then, inequality \eqref{mainineq} can be rewritten as
$$
\lambda_{1,g}({\rm B}_R(p))\leq \lambda_{1,\widetilde{g}}({\rm B}_R(p)).
$$
{Notice that in the statement of the Theorem \ref{MainThm} there are no conditions on the Ricci or sectional curvatures (as in the hypothesis of Cheng in \cite{Cheng2} and \cite{Cheng1}), neither on the mean curvature of the geodesic spheres (as in the hypothesis of Bessa and Montenegro in \cite{BM08}). But the geodesic spheres, in the critical metric tensor where equality in \eqref{mainineq} is attained, have radial mean curvature as in the result of Bessa and Montenegro. } 

{Since $(B_R,\widetilde{g})$ is a Riemannian model  of radius $R$ we can apply \eqref{eq:HMP} and hence
\begin{equation}\begin{aligned}
\lambda_{1,g}({\rm B}_R(p))\leq \lambda_{1,\widetilde{g}}({\rm B}_R(p))=&\lim_{k\to \infty}\left(\frac{\int_0^RT_k^2(t){\rm A}_g(t)\,dt}{\int_0^RT_{k+1}^2(t){\rm A}_g(t)\,dt}\right)^{{1/2}}\\
=&\lim_{k\to\infty}\frac{T_{k-1}(0)}{T_k(0)}=\lim_{k\to \infty}\frac{\int_0^RT_{k-1}(t){\rm A}_g(t)\,dt}{\int_0^RT_{k}(t){\rm A}_g(t)\,dt}\cdot
\end{aligned}
\end{equation}}

We must remark here that in the classical symmetrization results, the symmetrizated domain minimizes the first eigenvalue but in our result the ball with the rotationally symmetric metric tensor maximizes the first eigenvalue. Moreover, in the following example we show that our upper bounds can not be obtained from the classical comparison with the Ricci curvature neither {from} the comparison with the mean curvature.
\begin{example}\label{ex:boundl1noboundcurv}
	Let $(r,\theta)$ be polar coordinates in $\mathbb{R}^2$ around $\vec{0}\in \mathbb{R}^2$. Let us endow $\mathbb{R}^2$ with the following metric
	$$
	g=dr\otimes dr+ \left(r+\varphi(r) \cos(\theta)\right)^2d\theta\otimes d\theta,
	$$
	with
	$$
	\varphi: [0,\infty)\to [0,\infty),\quad t\mapsto \varphi(t):=\left\lbrace
	\begin{array}{lcl}
		0,&{\rm if}& t\leq 2,\\
		e^{-\frac{1}{(t-2)^2}},&{\rm if}& t>2.
	\end{array}
	\right.
	$$
	Hence,
	$$
	{\rm A}_g(t)=\int_0^{2\pi}\left(t+\varphi(t) \cos(\theta)\right)d\theta=2\pi t.
	$$
	By using the Theorem \ref{MainThm} we conclude that, for the geodesic disc of radius $R$,
	\begin{equation}\label{exbound}
		\lambda_{1,g}({\rm B}_R(\vec{0}))\leq\lambda_{1,g_{\rm can}}({\rm B}_R(\vec{0}))=\frac{j_0^2}{R^2}\approx \frac{5,78319}{R^2},
	\end{equation}
	where $g_{\rm can}$ is the canonical metric tensor in $\mathbb{R}^2$, $g_{\rm can}=dr\otimes dr+r^2\,d\theta\otimes d\theta$,  and $j_0$ is the first zero of the Bessel function $J_0$.  
	
	On the other hand, since the Ricci curvature is given by
	$$
	K(g)=\left\{
	\begin{array}{lcl}
		0,&{\rm for}& r\leq 2,\\
		\frac{2\,(3(r-4) r+10) \cos (\theta)}{(r-2)^6 \left(\cos (\theta)+e^{\frac{1}{(r-2)^2}} r\right)},&{\rm for} & r>2,
	\end{array}
	\right.
	$$
	for $r>2$, there are regions where $K(g)<0$. Hence, for a geodesic ball of radius $R>2$, the bound \eqref{exbound} can not be obtained by using the comparison of Cheng with the Ricci curvature  (where it is needed $K(g)\geq 0$).  Moreover, since the mean curvature ${\rm H}(t,\theta)$ pointing inward of the geodesic sphere ${\rm S}_t(\vec{0})$ is given by
	$$
 	{\rm H}(t,\theta)=\left\{
	\begin{array}{lcl}
		1/t,&{\rm for}& t\leq 2,\\
		1/t-\frac{(t-4) ((t-2) t+2) \cos (\theta)}{(t-2)^3 t \left(\cos (\theta)+e^{\frac{1}{(t-2)^2}} t\right)},&{\rm for} & t>2,
	\end{array}
	\right.
	$$
	for $t>2$ there are points in the sphere ${\rm S}_t(\vec{0})$ where ${\rm H}(t,\theta)>1/t$. Hence, for a geodesic ball of radius $R>2$, the bound \eqref{exbound} can not be obtained by using the comparison of Bessa and Montenegro with the mean curvature of the geodesic spheres  (where it is needed ${\rm H}(t,\theta)\leq 1/t$). Notice that, since for $R>2$ the mean curvature pointed inward of the geodesic spheres ${\rm S}_t(\vec{0})$ is not a radial function for any $t\in (2,R)$, then equality in \eqref{exbound} can not be obtained, namely
	\begin{equation}\label{exbound2}
		\lambda_{1,g}({\rm B}_R(\vec{0}))<\lambda_{1,g_{\rm can}}({\rm B}_R(\vec{0})) \quad{\rm for}\quad R>2.
	\end{equation}
	The above inequality allows us to state that in $(\mathbb{R}^2,g)$ there exists a domain $\Omega$ with symmetrizated radius $R(\Omega)>0$ with respect to the Euclidean space $(\mathbb{R}^2,g_{\rm can})$, \emph{i.e.}, with 
	$${\rm vol}_g(\Omega)={\rm vol}_{g_{\rm can}}(B_{R(\Omega)}(\vec{0}))=\pi R^2(\Omega)$$ 
	such that  
	$$
	{\rm vol}_{g}(\partial \Omega)<{\rm vol}_{g_{\rm can}}(\partial B_{R(\Omega)}(\vec{0}))=2\pi R(\Omega).
	$$
	Because otherwise, if for any domain $\Omega$, ${\rm vol}_{g}(\partial \Omega)\geq 2\pi {R(\Omega)}$, then by Faber-Krahn Theorem $\lambda_{1,g}({\rm B}_R(\vec{0}))$ should be greater or equal to $\lambda_{1,g_{\rm can}}({\rm B}_R(\vec{0}))$ but this is a contradiction with inequality \eqref{exbound2}. 
\end{example}

From Theorem \ref{MainThm}  we can compare the first eigenvalue of a geodesic ball of $(M,g)$ with the first eigenvalue of a geodesic ball of a real space form $\mathbb{M}^n_\kappa$ assuming certain behavior of the area function ${\rm A}_g(t)$ as follows
\begin{Theorem}\label{ChengTypeTheo}
	Let $(M,g)$ be a Riemannian manifold, let $p\in M$ be a point of $M$ with injectivity radius ${\rm inj}_{g}(p)$. Let ${\rm A}_{\kappa}(t)$ be {the} area function of the geodesic sphere of radius $t$ in $\mathbb{M}^n_\kappa$. Suppose that $R<{\rm inj}_g(p)$ and that for any $t<R$ the function
	$$
	t\mapsto\frac{{\rm A}_g(t)}{{\rm A}_{\kappa}(t)}
	$$
	is a decreasing function. Then, the first eigenvalue $\lambda_{1,g}({\rm B}_R(p))$ of the Laplacian for the Dirichlet problem on the geodesic ball ${\rm B}_R(p)$ of radius $R$ centered at $p$ is bounded by 
	\begin{equation}\label{cheng-ineq2}
		\lambda_{1,g}({\rm B}_R(p))\leq \lambda_1(R,\kappa),
	\end{equation}
	with equality in \eqref{cheng-ineq2} if and only if, for any $t\in(0,R)$, the mean curvature pointed inward ${\rm H}_{{\rm S}_t(p)}$ of the geodesic sphere ${\rm S}_t(p)$ is equal to the mean curvature pointed inward of the geodesic sphere of radius $t$ in $\mathbb{M}_\kappa^n$, namely
	$$
	{\rm H}_{{\rm S}_t(p)}=(n-1)\frac{S_\kappa'(t)}{S_\kappa(t)}, \quad {\rm for\; all }\quad 0<t< R,
	$$
	where
	$$
	S_\kappa(t):=\left\lbrace
	\begin{array}{rcl}
		\frac{\sin(\sqrt{\kappa}t)}{\sqrt{\kappa}},&{\rm if }&\kappa>0,\\
		t,&{\rm if }&\kappa=0,\\
		\frac{\sinh(\sqrt{-\kappa}t)}{\sqrt{-\kappa}},&{\rm if }&\kappa<0.
	\end{array}
	\right.
	$$
\end{Theorem}

Observe that when the Ricci curvatures of $M$ are bounded from below by {the Ricci curvatures of $\mathbb{M}_\kappa^n$, \emph{i.e.},} ${\rm Ric}_g\geq (n-1)\kappa$ (as the hypothesis of Cheng in  \cite{Cheng2}) the function 
$$
t\mapsto\frac{{\rm A}_g(t)}{{\rm A}_{\kappa}(t)}
$$
is a decreasing function for any $t<{\rm inj}_g(p)$ (for more details see \cite{Ch2}). Hence, our hypothesis about the behavior of the above quotient is weaker than the original hypothesis of Cheng but only up to the injectivity radius ${\rm inj}_g(p)$. Moreover, to characterize the equality, Cheng shows that the equality is attained if and only if the ball ${\rm B}_R(p)$ of $M$ is isometric to the ball with the same radius of $\mathbb{M}^n_\kappa$. {B}ut, with our hypothesis, equality is attained if and only if the mean curvature of the geodesic sphere ${\rm S}_t(p)$ is $(n-1)\frac{S_\kappa'(t)}{S_\kappa(t)}$ for all $t<R$.

Theorem \ref{ChengTypeTheo} is a particular case of Theorem \ref{ChengTypeTheoGen} of Section \ref{sec:ChengTypeTheo} obtained by using our Theorem \ref{MainThm}. Indeed, the general case is obtained when
$$
	t\mapsto\frac{{\rm A}_g(t)}{{\rm A}_{W}(t)}
$$
is decreasing where ${\rm A}_{W}(t)$ is the volume of the geodesic sphere of radius $t$ in some Riemannian model.

\subsection*{Acknowledgments}
The authors are grateful to professor Vicente Palmer for his valuable help and for his useful comments and suggestions during the preparation of the present paper.

%
%
%
%
%
\section{Volume-based rotational symmetrization of the metric tensor}\label{rotmetcomp}
The main tool of this paper consists in the transformation of the metric tensor $g$ of a Riemannian manifold $(M,g)$ to a metric tensor $\widetilde{g}$ such that $({\rm B}_R(p),\widetilde{g})$ is rotationally symmetric and the volume of the geodesic spheres of $({\rm B}_R(p),\widetilde{g})$ coincides with the volume of geodesic spheres of $({\rm B}_R(p),g)$. In this section we show in Theorem \ref{prop2.1} that this new {metric tensor} $\widetilde{g}$ is smooth, in Proposition \ref{malsagxulo} we compare the behavior of the distance function with respect $g$ and $\widetilde{g}$, as well as the area function for geodesic spheres and  the expression of the Laplacian with respect to the new metric tensor $\widetilde{g}$. Finally, in Corollary \ref{corollarifinal} we characterize the first eigenfunction and first eigenvalue of $({\rm B}_R(p),\widetilde{g})$.

\subsection{Definition and smoothness of the rotationally symmetric metric {tensor} of comparison}
Let $(M,g)$ be a Riemannian manifold, let $p\in M$ be a point in $M$, let ${\rm B}_R(p)$  be the geodesic ball of $M$ centered at $p$ of radius $R<{\rm inj}_g(p)$, and let us denote by
$$
\mathcal{B}_R(\vec{0}):=\left\{v\in T_pM\, :\, g(v,v)<R^2\right\}.
$$
The exponential map $\exp_p: T_pM\to M$ induces a diffeomorphism from  $\mathcal{B}_R(\vec{0})$ to ${\rm B}_R(p)$. Let $\{E^i\}_{i=1}^n$ be an orthonormal basis of $T_pM$ and let us denote {by $x^1,\ldots,x^n$ the \emph{normal coordinate functions} associated to the orthonormal basis  $\{E^i\}_{i=1}^n$ with respect to $g$ given by $x^i(q):=g( E^i,\exp_p^{-1}(q))$. T}hen the map 
\begin{equation}\label{eq:diffzeta}
 	\zeta:{\rm B}_R(p)\to \mathbb{D}_R(\vec{0}),  \quad
	q\mapsto  \zeta(q):=\left(x^1(q),\ldots,x^n(q)\right)\quad 
\end{equation}  
is a diffeomorphism from ${\rm B}_R(p)$ to the open Euclidean ball  $\mathbb{D}_R(\vec{0})$, 
$$
\mathbb{D}_R(\vec{0}):=\left\{x\in \mathbb{R}^n\, :\, (x^1)^2+\cdots +(x^n)^2<R^2\right\}.
$$
Since $\zeta=(x^1,\ldots,x^n):{\rm B}_R(p)\to \mathbb{D}_R(\vec{0})$ is a coordinate system (or chart) in ${\rm B}_R(p)$ then, for any $q\in {\rm B}_R(p)$, the coordinate vectors
$$
\left\{\left.\frac{\partial}{\partial x^1}\right\vert_q,\ldots ,\left.\frac{\partial}{ \partial x^n}\right\vert_q \right\}
$$
form a basis for the tangent space $T_q({\rm B}_R(p))$ (see for instance Theorem 12 of \cite{O'N}). For each $1\leq i\leq n$ the vector field $\frac{\partial }{\partial x^i}$ on ${\rm B}_R(p)$ sending each $q$ to $\left.\frac{\partial}{ \partial x^i}\right\vert_q$ is  called the \emph{$i$th coordinate vector field} and the one-forms 
$$
\left\{dx^1,\ldots,dx^n\right\} \quad{\rm given}\; {\rm by}\quad dx^i\left(\frac{\partial}{\partial x^j}\right)={\delta_{i,j}}
$$  
are called \emph{coordinate one-forms}. Moreover, for any smooth function $f:{\rm B}_R(p)\to\mathbb{R}$, the one-form $df$ can be obtained as
$$
df=\sum_{i=1}^n\frac{\partial f}{\partial x^i}dx^i.
$$
Let us introduce the following functions $r$ and $\pi$ associated $\zeta$ 
\begin{equation}\label{rpifunctions}
	\begin{array}{ll}
		r:{\rm B}_R(p)\to \mathbb{R},& q\mapsto r(q):=\sqrt{(x^1(q))^2+\cdots +(x^n(q))^2},\vphantom{\dfrac{1}{1}}\\
		\pi: {\rm B}_R(p)-\{p\}\to \mathbb{S}^{n-1}_1,& q\mapsto \pi(q):=\dfrac{\zeta(q)}{r(q)}.
	\end{array}
\end{equation}
In the following definition we introduce the rotationally symmetric metric tensor of comparison by using the above functions.
	
\begin{definition}[Rotationally symmetric metric {tensor} of comparison]\label{def:rotsymmodspaofcom}
	Let $(M,g)$ be a $n$-dimensional Riemannian manifold. Let ${\rm B}_R(p)$ be the geodesic ball of radius $R$ centered at $p {\in M}$. Suppose that $R<{\rm inj}_g(p)$.  Let  $\{E^i\}_{i=1}^n$ be an orthonormal basis of $T_pM$ and let $\zeta=(x^1,\ldots,x^n):{\rm B}_R(p)\to \mathbb{D}_R(\vec{0})$  be the normal coordinate functions associated to  $\{E^i\}_{i=1}^n$. The \emph{rotationally symmetric metric {tensor} of comparison} $\widetilde{g}$ associated to $g$ is the metric tensor given by
	\begin{equation}\widetilde{g}=\left\lbrace\begin{array}{lcl}
		dr\otimes dr+(\omega_g^2\circ r)\pi^*g_{\mathbb{S}^{n-1}_1}, & {\rm on} & {\rm B}_R(p)-\{p\},\\
		\displaystyle\sum_{i=1}^ndx^i\otimes dx^i, & {\rm on} & p,
		\end{array}\right.
	\end{equation}
	where $r,\pi$ are given by \eqref{rpifunctions}, $\pi^*g_{\mathbb{S}^{n-1}_1}$ is the pullback by $\pi$ of the canonical metric tensor $g_{\mathbb{S}^{n-1}_1}$ of $\mathbb{S}^{n-1}_1$, \footnote{Recall that the canonical metric tensor $g_{\mathbb{S}^{n-1}_1}$ of $\mathbb{S}^{n-1}_1$ is the metric tensor that inherits $\mathbb{S}^{n-1}_1:=\{x\in\mathbb{R}^n\, :\, \sum_{i=1}^n(x^i)^2=1\}$ when is considered as a submanifold of $\mathbb{R}^n$ with the canonical metric tensor $\sum_{i=1}^ndx^i\otimes dx^i$.} and $\omega_g:[0,R)\to\mathbb{R}_+$ is the positive function given by 
	\begin{equation}\label{eq:relationgeosphvol}
		t\mapsto\omega_g(t):=\left(\frac{{\rm A}_g(t)}{{\rm vol}\left({\mathbb{S}_1^{n-1}}\right)} \right)^{\frac{1}{n-1}}
	\end{equation}
	where $A_g(t)$ is the volume of the sphere $S_t(p)$ of radius $t$ centered at $p$, \emph{i.e.},
$${\rm A}_g(t):={\rm vol}_g({\rm S}_t(p)).$$

\end{definition}
The main result of this section is to show that the metric tensor $\widetilde{g}$ is well defined. Indeed, in the following Theorem we will prove the smoothness of this new metric tensor.
\begin{theorem}\label{prop2.1}
	Let $(M,g)$ be a $n$-dimensional Riemannian manifold. Let ${\rm B}_R(p)$ be the geodesic ball of radius $R$ centered at $p{\in M}$. Suppose that $R<{\rm inj}_g(p)$. Then, the rotationally symmetric metric {tensor} of comparison $\widetilde{g}$ associated to $g$ is smooth in ${\rm B}_R(p)$.
\end{theorem}
\begin{proof}
	First, we are proving that $\omega_g(t)$ can be rewritten as
	\begin{equation*}
		\omega_g(t)=t\left(1+t^2\varphi(t^2)\right)
	\end{equation*} 
	with some positive smooth function $\varphi$. The area function ${\rm A}_g(t)$ is a smooth function up to its injectivy radius (see \cite{Ch2}) and it has Taylor expansion about $t=0$ given by (see Theorem 3.1 of \cite{Gray})
	\begin{equation*}
		{\rm A}_g(t)=a_0t^{n-1}+a_2t^{n+1}+a_4t^{n+3}+\cdots
	\end{equation*}	
	for some constants $a_{2k}\in\mathbb{R}$, $k\in\mathbb{N}$, with $a_0={\rm vol}\left(\mathbb{S}^{n-1}\right)$. In particular, ${\rm A}_g(0)=0$, the derivatives ${\rm A}_g^{(k)}(0)=0$ for $k=1,\dots,n-2$, and the derivatives ${\rm A}_g^{(n+2k)}(0)=0$ for $k\in\mathbb{Z}$. Since every derivative of ${\rm A}_g(t)$ vanishes up to $n-1$ order and since  ${\rm A}_g(t)$ is a smooth function up to $t={\rm inj}_g(p)$, then  we can use for all $t\in [0, {\rm inj}_g(p)]$ the Taylor expansion with integral form of the remainder (see \cite{Spivak} for instance) and we can rewrite ${\rm A}_g(t)$ as
	$$
	{\rm A}_g(t)=\dfrac{1}{(n-2)!}\int_0^t(t-x)^{n-2}{\rm A}_g^{(n-1)}(x)\,dx.
	$$
	By {using } the change of variable   $x=st$ in the above expression, we can express
	$$
	{\rm A}_g(t)=a_0\,t^{n-1}f(t),\quad f(t):=\dfrac{1}{a_0 (n-2)!}\int_0^1(1-s)^{n-2}{\rm A}_g^{(n-1)}(st)\,ds.
	$$
	The function $f(t)$ is a positive smooth function with 
	$$
	f^{(k)}(t)=\frac{1}{a_0 (n-2)!}\int_0^1(1-s)^{n-2}s^k{\rm A}_g^{(n-1+k)}(st)\,ds.
	$$
	In particular, $f(0)=1$ and, since ${\rm A}_g^{(n+2k)}(0)=0$ for $k\in\mathbb{Z}$, the odd order derivatives of $f(t)$ vanish at $0$. In fact,
	$$
	f^{(2k+1)}(0)=\frac{1}{a_0 (n-2)!}\int_0^1(1-s)^{n-2}s^k{\rm A}_g^{(n+2k)}(0)\,ds=0
	$$
	for all $k\in\mathbb{N}$. Then  $f$ can be extended to a smooth even function $\widetilde{f}(t)$ with $\widetilde{f}(0)=1$, given by
	$$
	\widetilde{f}(t):=\left\lbrace
	\begin{array}{lcr}
		f(t ), & {\rm if} & t\geq 0, \\
		f(-t), & {\rm if} & t< 0.
	\end{array}
	\right.
	$$
	Since $\widetilde{f}$ is a smooth even function, it can be expressed as (see \cite{Whitney43})
	$$
	\widetilde{f}(t)=h(t^2)
	$$
	with a positive smooth function $h$. Notice that $h(0)=f(0)=1$. We can therefore express the area function  as
	\begin{equation}\label{eq:expresionareagfunction2}
		{\rm A}_g(t)=a_0\,t^{n-1}h(t^2)={\rm vol}\left(\mathbb{S}^{n-1}\right)t^{n-1}h(t^2).
	\end{equation}
	Now we define the function $F(t):=\left(h(t)\right)^{\frac{1}{n-1}}$. Then, since $h(t)>0$ on $0\leq t<{{\rm inj}_g(p)}$ and $h(0)=1$, then $F$ is smooth with $F(0)=1$. Hence, from \eqref{eq:relationgeosphvol}, we obtain that
	\begin{equation*}
		\omega_g(t)=t\,F(t^2).
	\end{equation*}
	On the other hand, since $F$ is a positive smooth function with $F(0)=1$ then we can express $F$ as
	$$
	F(t)=1+\int_0^tF'(x)\,dx=1+t\int_0^1F'(st)\,ds
	$$
	Thus, we can rewrite $F$ as
	$$
	F(t)=1+t\varphi(t),\quad \varphi(t):=\int_0^1F'(st)\,ds.
	$$	
	This implies that $F(t^2)=1+t^2\varphi(t^2)$ and 
	\begin{equation}\label{eq:warpingexpresio2}
		\omega_g(t)=t\left(1+t^2\varphi(t^2)\right).
	\end{equation}
	Finally, to prove the {smoothness} of the  rotationally symmetric metric {tensor} of comparison observe that since on ${\rm B}_R(p)-\{p\}$
$$
\sum_{i=1}^ndx^i\otimes dx^i=dr\otimes dr+r^2\pi^*g_{\mathbb{S}^{n-1}_1}\quad {\rm and}\quad dr=\sum_{i=1}^n\frac{x^i}{r}dx^i,
$$
the metric tensor of comparison $\widetilde{g}$  can be expressed in normal coordinates with respect to any orthonormal basis $\{E_i\}_{i=1}^n$ as
	\begin{equation}\label{cartesian}
		\begin{array}{lcr}
			 \widetilde{g}&=& \displaystyle\sum_{i,j=1}^n\left(\delta_{i\, j}+\dfrac{\omega_g^2(r)-r^2}{r^4}\left(r^2\,\delta_{i\,j}-x^ix^j\right)\right)\,dx^i\otimes dx^j.
		\end{array}
	\end{equation}
	Then, applying equation \eqref{eq:warpingexpresio2},
	\begin{equation}
		\begin{array}{lcl}
			\widetilde{g}&=& \displaystyle\sum_{i,j=1}^n\left(\delta_{i\, j}+\dfrac{r^2\left(1+r^2\varphi(r^2)\right)^2-r^2}{r^4}\left(r^2\,\delta_{i\,j}-x^ix^j\right)\right)\,dx^i\otimes dx^j\\	&=&\displaystyle\sum_{i,j=1}^n\left(\vphantom{\dfrac{1}{1}}\delta_{i\, j}+\left(2\varphi(r^2)+r^2\varphi^2(r^2)\right)\left(r^2\,\delta_{i\,j}-x^ix^j\right)\right)\,dx^i\otimes dx^j.
		\end{array}
	\end{equation}
	Since $x_i, r^2$ are smooth functions from ${\rm B}_R(p)$ to $\mathbb{R}$ and $\varphi$ is a smooth function of $\mathbb{R}_+$ the Theorem follows.

\end{proof} 
		
Observe from the proof of Theorem that what is proved is that $\omega_g(t)=t\,F(t^2)$ with $F(t)$ being a smooth function satisfying $F(0)=1$. Moreover in the proof of the Theorem is proved that any metric tensor on ${\rm B}_R(p)-\{p\}$  of the form
\begin{equation}\label{polarmetric}
	dr\otimes dr+(\omega^2\circ r)\pi^*g_{\mathbb{S}^{n-1}_1} 
\end{equation}
with positive warping function $\omega(t)$, can be extended to a smooth metric tensor in ${\rm B}_R(p)$ if there exists a smooth function $F(t)$ with $F(0)=1$ such that $\omega(t)=t\,F(t^2)$. This is equivalent to the classical condition on the warping function $\omega$, which is
\begin{equation}\label{eq:warpingfunctionconditionsclassic}
	\omega(0)=0,\quad \omega'(0)=1 \quad\text{and} \quad \omega^{(2k)}(0)=0\quad\text{for all}\quad k\in\mathbb{N}^*,
\end{equation}
where $\omega^{(2k)}$ denotes all the even derivatives of $\omega$. Indeed, assuming $\omega(t)=t\,F(t^2)$  with $F(0)=1$ it is easy to check that $\omega$ satisfies \eqref{eq:warpingfunctionconditionsclassic}. In the other direction, if $\omega$ satisfies condition \eqref{eq:warpingfunctionconditionsclassic} we can construct an even smooth function $f:\mathbb{R}\to\mathbb{R}_+$ such that $\omega(t)=tf(t)$ and then, using \cite{Whitney43} as in the proof of Theorem \ref{prop2.1}, we obtain $f(t)=F(t^2)$.

\subsection{Properties of the rotationally symmetric metric {tensor} of comparison}\label{secpropmet}

In this subsection of the paper we show some properties of the rotationally symmetric metric {tensor} of comparison $\widetilde{g}$. First of all, we will clarify in which sense the metric tensor $\widetilde{g}$ is rotationally symmetric. The Orthogonal group ${\rm O}(n)=\{R\in {\rm GL}(n)\, :\, R^TR=RR^T=\mathfrak{1}_n\}$ acts on $\mathbb{D}_R(\vec{0})$ by 
$$
{\rm O}(n)\times \mathbb{D}_R(\vec{0})\to \mathbb{D}_R(\vec{0}),\quad (R,x)\mapsto R\, x.
$$
By using the diffeomorphism $\zeta : {\rm B}_R(p)\to \mathbb{D}_R(\vec{0})$, defined in \eqref{eq:diffzeta}, we can define the action of ${\rm O}(n)$ on ${\rm B}_R(p)$, by
$$
{\rm O}(n)\times {\rm B}_R(p)\to {\rm B}_R(p),\quad (R,q)\mapsto \zeta^{-1}(R\, \zeta(q)).
$$
When ${\rm B}_R(p)$ is endowed with the metric tensor $\widetilde{g}$, the group ${\rm O}(n)$ acts by isometries. Since $\widetilde{g}$ remains invariant under the action of the Orthogonal group we will say that the metric tensor $\widetilde{g}$ is \emph{rotationally symmetric}.

In \cite{grigoryan-book} is defined an $n$-dimensional Riemannian manifold $(M,g)$ as a \emph{Riemannian model} if the following conditions are satisfied:
\begin{enumerate}
	\item There is a chart of $M$ that covers all $M$, and the image of this chart in $\mathbb{R}^n$ is a ball $\mathbb{D}_{R_0}$ of radius $R_0\in (0,+\infty]$.
	
	\item The metric {tensor} $g$ in the polar coordinates $(r,\theta)$ in the above chart has the form given by \eqref{polarmetric} with $\omega$ a positive function.
\end{enumerate}
The number $R_0$ is called the radius of the model $M$.  Observe that given a Riemannian model $(M,g)$ the metric tensor $g$ is rotationally symmetric in any geodesic ball of radius $R<R_0$. On the other hand, given a Riemannian manifold $(M,g)$, the geodesic ball $({\rm B}_R(p),\widetilde{g})$ of radius $R<{\rm inj}_g(p)$ endowed with the rotationally symmetric metric {tensor} of comparison $\widetilde{g}$ associated to $g$ is a Riemannian model of radius $R$. 

The expression of the distance function, area function, and Laplacian of functions with respect to $g$ and $\widetilde{g}$ are given in the following
\begin{proposition}\label{malsagxulo}
	Let $(M,g)$ be a $n$-dimensional Riemannian manifold. Let ${\rm B}_R(p)$ be the geodesic ball of radius $R<{\rm inj}_g(p)$ centered at $p{\in M}$, let $\exp_p:T_pM\to M$ be the exponential map associated to $g$, let  $\widetilde{g}$ be the rotationally symmetric metric {tensor} of comparison  associated to $g$.  Then for any $q\in {\rm B}_R(p)$,
	\begin{enumerate}
		\item $\displaystyle\nabla r(q)=\widetilde{\nabla}r(q)=\partial r(q)=\sum_{i=1}^n\frac{x^i(q)}{r(q)}\left.\frac{\partial}{\partial x^i}\right\vert_q$,
		\bigskip
		
		\item $r(q):=\Vert \exp^{-1}(q)\Vert={\rm dist}_g(p,q)={\rm dist}_{\widetilde{g}}(p,q)$,
		\bigskip
		
		\item $g( \nabla r(q),\nabla r(q))=\widetilde{g}(\widetilde{\nabla}r(q),\widetilde{\nabla}r(q))=1$,
		\bigskip
		
		\item $ {\rm A}_{\widetilde{g}}(t)={\rm vol}_{\widetilde{g}}(S_t(p))={\rm vol}\left(\mathbb{S}^{n-1}_1\right)\omega_g^{n-1}(t)={\rm A}_g(t)$,
        \bigskip
        
		\item For any smooth function $f:{\rm B}_R(p)\to \mathbb{R}$,
		\begin{equation}\label{eq:laprotsym}
		    \Delta_{\widetilde g} f=(n-1)\dfrac{\omega'_g(r)}{\omega_g(r)}\dfrac{\partial f}{\partial r}+\dfrac{\partial^2f}{\partial r^2}+\dfrac{1}{\omega^2_g(r)}\Delta_{\mathbb{S}_1^{n-1}}(f\circ\pi^{-1})\circ \pi,
		\end{equation}
	\end{enumerate}
	where $\nabla$ and $\widetilde{\nabla}$ denote the gradient with respect to $g$ and $\widetilde{g}$ respectively and $\Delta_{\mathbb{S}_1^{n-1}}$ denotes the Laplacian of the $(n-1)$-dimensional usual unit sphere.
\end{proposition}

Since $({\rm B}_R(p),\widetilde{g})$ is a Riemannian model, before to prove the Theorem \ref{MainThm} we will need the following consideration about  first eigenfunction and the first eigenvalue for the Dirichlet problem in a geodesic ball of a Riemannian model. 

\begin{proposition}\label{prop-first-eigenvalue}
	Let $(M,g_\omega)$ be a $n$-dimensional Riemannian model, let ${\rm B}_R(p)$ be a geodesic ball centered at $p\in M$ with radius $R<{\rm inj}_g(p)$. Suppose that the smooth metric tensor  on ${\rm B}_R(p)-\{p\}$ is given by
	\begin{equation*}
		g_\omega=dr\otimes dr+(\omega^2\circ r)\pi^*g_{\mathbb{S}^{n-1}_1}
	\end{equation*}
	with $\omega:[0,R)\to\mathbb{R}_+$ a positive function. Then, any positive first eigenfunction $\phi_1$ of the Laplacian $\Delta_{g_\omega}$ for the Dirichlet problem on ${\rm B}_R(p)$  is radial, $\phi_1(q)=f_1(r(q))$ with $f_1$ a smooth function  such that
	\begin{equation*}
		f_1'(0)=0\; {\rm and }\; f_1'(t)<0\; {\rm for }\; t\in(0,R].
	\end{equation*}
\end{proposition}
\begin{proof}
	If $\phi_1$ is any positive first eigenfunction of the Laplacian $\Delta_{g_\omega}$ for the Dirichlet problem on a rotationally symmetric geodesic ball ${\rm B}_R(p)$ {then}, since $g_w$ is rotationally symmetric {,} $\phi_1$ is a radial function,  {\emph{i.e.}}, we can rewrite, for all $q\in{\rm B}_R(p)$, the first eigenfunction as $\phi_1(q)=f_1(r(q))$ where $f_1$ is a positive real valued smooth function. Moreover, an easy computation leads to
    $$
    \Delta_{g_w}\phi_1(q)=(n-1)\dfrac{\omega'(r(q))}{\omega(r(q))}f_1'(r(q))+f_1''(r(q))=-\lambda_1\left({\rm B}_R(p)\right)f_1(r(q)),
    $$
    where $\lambda_1\left({\rm B}_R(p)\right)>0$ is the first eigenvalue of the Laplacian for the Dirichlet problem in ${\rm B}_R(p)$. Hence, for any $t\in [0,R]$,
    \begin{equation}\label{eq:laprotsym2}
    	(n-1)\dfrac{\omega'(t)}{\omega(t)}f_1'(t)+f_1''(t)=-\lambda_1\left({\rm B}_R(p)\right)f_1(t)
    \end{equation}
    By using that $\phi_1$ is a radial function  it is know that $f_1'(0)=0$ (see \cite{Ch1} for instance).  { On the other hand, from \eqref{eq:laprotsym2}	 we know that  all the critical points of $f_1$ are relative maximums. Since  for a real valued smooth function between two relative maximums there is at least one relative minimum, therefore $f_1$ can only have one maximum in $[0,R)$. Thus,  $0$ is the only critical point of $f_1$ and then, $f_1$ is a decreasing function in $(0,R)$. Namely, $f_1'(t)<0$ for all $t\in(0,R)$.}\end{proof}
	
\begin{theorem}[\cite{BGJ19}]\label{BGJthm} 
	Let $(M,g_\omega)$ be a $n$-dimensional Riemannian model, let ${\rm B}_R(p)$ be a geodesic ball centered at $p$ with radius $R<{\rm inj}_{g_\omega}(p)$. Suppose that the smooth metric tensor  on ${\rm B}_R(p)-\{p\}$ is given by
	\begin{equation}
		g_\omega=dr\otimes dr+(\omega^2\circ r)\pi^*g_{\mathbb{S}^{n-1}_1}
	\end{equation}
	with $\omega:[0,R)\to\mathbb{R}_+$ a positive function. Then, the first eigenvalue $\lambda_{1,g_\omega}({\rm B}_R(p))$ of the Laplacian $\Delta_{g_\omega}$ for the Dirichlet problem on ${\rm B}_R(p)$ is given by
	$$
	\lambda_{1,g_\omega}({\rm B}_R(p))= \lim_{k\to \infty}\frac{\left\lVert T_k(t)\right\lVert_{{2}}}{\left\lVert T_{k+1}(t)\right\lVert_{{2}}}
	$$
	where $\lVert\cdot\lVert_{{2}}$ denotes the $L_2$-norm on ${\rm B}_R(p)$ and with 
	$$
	\left\{
	\begin{array}{rcl}
		T_0(t)&=&1,\\
		T_k(t)&=&\int_t^R\frac{\int_0^\sigma T_{k-1}(s)\,\omega^{n-1}(s)\,ds}{\omega^{n-1}(\sigma)}\,d\sigma.
	\end{array}
	\right.
	$$
\end{theorem}
	
Since $({\rm B}_R(p),\widetilde{g})$ is a particular case of Riemannian model with $\omega=\omega_g${,} from Proposition \ref{prop-first-eigenvalue} and Theorem \ref{BGJthm}{,} we can state the following
\begin{corollary}\label{corollarifinal}
    Let $(M,g)$ be a $n$-dimensional Riemannian manifold. Let ${\rm B}_R(p)$ be the geodesic ball of radius $R<{\rm inj}_g(p)$ centered at $p\in M$, let  $\widetilde{g}$ be the rotationally symmetric metric {tensor} of comparison associated to $g$. Then, any positive first eigenfunction $\phi_1$ of the Laplacian $\Delta_{\widetilde{g}}$ for the Dirichlet problem on ${\rm B}_R(p)$ is radial, $\phi_1(q)=f_1(r(q))$ with $f_1$ a smooth function such that
	\begin{equation}\label{eq:conditionsradialf1}
		f_1'(0)=0\; {\rm and }\; f_1'(t)<0\; {\rm for}\; t\in(0,R]
	\end{equation}
	and moreover, the first eigenvalue is given by
	$$
	\lambda_{1,\widetilde{g}}({\rm B}_R(p))=\lim_{k\to \infty}\left(\frac{\int_0^RT_k^2(t){\rm A}_g(t)\,dt}{\int_0^RT_{k+1}^2(t){\rm A}_g(t)\,dt}\right)^{{1/2}}
	$$
	with
	$$
	\left\{
	\begin{array}{rcl}
		T_0(t)&=&1,\\
		T_k(t)&=&\int_t^R\frac{\int_0^\sigma {T_{k-1}(s){\rm A}_g(s)}\,ds}{{\rm A}_g(\sigma)}\,d\sigma.
	\end{array}
	\right.
	$$
\end{corollary}

\section{Proof of the Theorem \ref{MainThm}}\label{comparsec}

In this section we will prove Theorem \ref{MainThm}. Indeed, in this section we will prove Theorem \ref{th:cotaprimervalorpropiperdalt} which is nothing else than Theorem \ref{MainThm} but showing in \eqref{Eq:Comparison} the comparison of the first eigenvalues $\lambda_{1,g}\left({\rm B}_R(p)\right)$ and $\lambda_{1,\widetilde{g}}\left({\rm B}_R(p)\right)$ and the equality 
$$
\lambda_{1,\widetilde{g}}\left({\rm B}_R(p)\right)=\lim_{k\to \infty}\left(\frac{\int_0^RT_k^2(t){\rm A}_g(t)\,dt}{\int_0^RT_{k+1}^2(t){\rm A}_g(t)\,dt}\right)^{{1/2}},
$$
given in Corollary \ref{corollarifinal}.
\begin{theorem}\label{th:cotaprimervalorpropiperdalt}
	Let $(M,g)$ be a Riemannian manifold, let $p\in M$ be a point of $M$ with injectivity radius ${\rm inj}_g(p)$, and let ${\rm B}_R(p)$ be the geodesic ball of radius $R$ centered at $p$. Suppose $R<{\rm inj}_g(p)$, then the first eigenvalue $\lambda_{1,g}\left({\rm B}_R(p)\right)$ of the Laplacian for the Dirichlet problem on ${\rm B}_R(p)$ is bounded by
	\begin{equation}\label{Eq:Comparison}
		\lambda_{1,g}\left({\rm B}_R(p)\right)\leq\lambda_{1,\widetilde{g}}\left({\rm B}_R(p)\right)=\lim_{k\to \infty}\left(\frac{\int_0^RT_k^2(t){\rm A}_g(t)\,dt}{\int_0^RT_{k+1}^2(t){\rm A}_g(t)\,dt}\right)^{{1/2}}{,}
	\end{equation}
	where $\widetilde{g}$ is the rotationally symmetric metric {tensor} of comparison associated to $g$. Futhermore, equality is attained in \eqref{Eq:Comparison} if and only if, for any $t\in (0,R)$, the mean curvature pointed inward ${\rm H}_{{\rm S}_t(p)}$ of the geodesic sphere ${\rm S}_t(p)$ of radius $t$ centered at $p$ is a radial function. Namely, equality in \eqref{Eq:Comparison} is attained if and only if there exists a smooth function $h(t)$ such that
	$$
	{\rm H}_{{\rm S}_t(p)}=h(t)  \quad {\rm for\;{any} }\quad 0<t<R.
	$$
\end{theorem}
\begin{proof}
	To obtain upper bounds for the first eigenvalue we compute the Rayleigh quotient (see \cite{Ch1} for more information about the Rayleigh quotient) with respect to $({\rm B}_R(p),g)$ of the first eigenfunction on $({\rm B}_R(p),\widetilde{g})$, where $\widetilde{g}$ is the rotationally symmetric metric {tensor} of comparison defined in Section \ref{rotmetcomp}.
		
	Since $R<{\rm inj}_g(p)$, the exponential map $\exp_p: \mathcal{B}_R(\vec{0})\to {\rm B}_R(p)$ induces a diffeomorphism $\zeta:{\rm B}_R(p)\to \mathbb{D}_R(\vec{0})$ (see Section \ref{rotmetcomp}  for the definition of the the diffeomorphism $\zeta$). The metric tensor $g$ in ${\rm B}_R(p)-\{p\}$ can be expressed as
	$$
	g=dr\otimes dr+\sum_{i,j=1}^{n-1}G_{i,j}(r,\theta)\,d\theta^i\otimes d\theta^j,
	$$
	for some positive definite matrix $G_{i,j}$, where we have used the coordinate system 
	$$
	\begin{array}{lcl}
	q\mapsto (r(q),\theta^1(q), \cdots \theta^{n-1}(q)):=(r(q),\widetilde{\theta}^1(\pi(q)),\cdots,\widetilde{\theta}^{n-1}(\pi(q)))
	\end{array}
	$$	
	with the maps $q\mapsto r(q)=\Vert \zeta(q)\Vert$ and $q\mapsto \pi(q)=\frac{\zeta(q)}{r(q)}$ and $\{\widetilde{\theta^i}\}_{i=1}^{n-1}$ being a system of local coordinates on the sphere $\mathbb{S}^{n-1}_1$. Using this coordinates, the Riemannian volume element can be obtained as
	$$
	d{\rm V}_g=\sqrt{\det(G(r,\theta))}\,dr\wedge d\theta^1\wedge\cdots\wedge d\theta^{n-1}.
	$$
	The {area function of the geodesic} sphere ${\rm S}_t(p)$ of radius $t$ centered at $p$ is  
	$$
	{\rm A}_g(t)=\int_{\mathbb{S}^{n-1}_1}\sqrt{\det(G(t,\theta))}\,d\theta^1\wedge\cdots\wedge d\theta^{n-1}
	$$
	and the mean curvature vector field of the geodesic sphere ${\rm S}_t(p)$ of radius $t$ centered at $p$ can be expressed as
	$$
	\vec{\rm H}_{S_t(p)}=-\left.\frac{\partial}{\partial s}\ln{\sqrt{\det(G(s,\theta))}}\right\vert_{s=t}\partial r.
	$$
	Now, we will make use of the rotationally symmetric metric {tensor} of comparison $\widetilde{g}$ associated to $g$ using the warping function $\omega_g(t)=\left(\frac{{\rm A}_g(t)}{{\rm vol}\left({\mathbb{S}_1^{n-1}}\right)}\right)^{\frac{1}{n-1}},$ as 
	\begin{equation}\widetilde{g}=\left\lbrace\begin{array}{lcl}
		dr\otimes dr+(\omega_g^2\circ r)\pi^*g_{\mathbb{S}^{n-1}_1}, & {\rm on} & {\rm B}_R(p)-\{p\},\\
		\displaystyle\sum_{i=1}^ndx^i\otimes dx^i, & {\rm on} & p.
		\end{array}\right.
	\end{equation}
	The first positive eigenfunction $\phi_1$ and the first eigenvalue $\lambda_{1,\widetilde{g}}({\rm B}_R(p))$ of the Laplacian $\Delta_{\widetilde{g}}$ for the Dirichlet problem on ${\rm B}_R(p)$ are related by the Rayleigh quotient with respect to $\widetilde{g}$ 
	\begin{equation}\label{eq:Rayquo1}
		\lambda_{1,\widetilde{g}}({\rm B}_R(p))=\frac{\int_{{\rm B}_R(p)}\widetilde{g}(\widetilde{\nabla} \phi_1(q),\widetilde{\nabla} \phi_1(q))\,d{\rm V}_{\widetilde{g}}(q)}{\int_{{\rm B}_R(p)}\phi_1^2(q)\,d{\rm V}_{\widetilde{g}}(q)}.
	\end{equation}
	Our upper bound for $\lambda_{1,g}({\rm B}_R(p))$ in \eqref{Eq:Comparison} is obtained by using the Rayleigh quotient with respect to $g$ of the first eigenfunction $\phi_1$ with respect to $\widetilde{g}$, \emph{i.e.},
	\begin{equation}\label{eq:Rayquo2}
		\lambda_{1,g}({\rm B}_R(p))\leq \frac{\int_{{\rm B}_R(p)}g(\nabla \phi_1(q),\nabla \phi_1(q))\,d{\rm V}_{g}(q)}{\int_{{\rm B}_R(p)}\phi_1^2(q)\,d{\rm V}_{g}(q)}.
	\end{equation}
	But, the right hand side of the above inequality can be expressed as 
	\begin{equation}\label{eq:Rayquo3}
		\frac{\int_{{\rm B}_R(p)}g(\nabla \phi_1(q),\nabla \phi_1(q))\,d{\rm V}_{g}(q)}{\int_{{\rm B}_R(p)}\phi_1^2(q)\,d{\rm V}_{g}(q)}=\frac{\int_{{\rm B}_R(p)}\widetilde{g}(\widetilde{\nabla} \phi_1(q),\widetilde{\nabla} \phi_1(q))\,d{\rm V}_{\widetilde{g}}(q)}{\int_{{\rm B}_R(p)}\phi_1^2(q)\,d{\rm V}_{\widetilde{g}}(q)}
	\end{equation}
	because taking into account that $\phi_1(q)=f_1(r(q))$ is a decreasing radial function (see Corollary \ref{corollarifinal}) and by using the co-area formula (see \cite{Ch1}) and Proposition \ref{malsagxulo} 
	\begin{align*}
		\int_{{\rm B}_R(p)}\phi_1^2(q)\,d{\rm V}_g(q) & =\int_{{\rm B}_R(p)}\dfrac{f^2_1(r(q))}{g\left(\nabla r(q),\nabla r(q)\right)}\,g\left(\nabla r(q),\nabla r(q)\right)\,d{\rm V}_g(q)\\
		& = \int_0^R\left(\int_{\left\lbrace q\in M\,:\,r(q)=t\right\rbrace}\dfrac{f_1^2(r(q))}{g\left(\nabla r(q),\nabla r(q)\right)}\,d{\rm A}_g(q)\right)dt\\
		& = \int_0^R f_1^2(t) {\rm A}_g(t)\,dt=\int_0^R f_1^2(t) {\rm A}_{\widetilde{g}}(t)\,dt\\
		& = \int_0^R\left(\int_{\left\lbrace q\in M\,:\, r(q)=t\right\rbrace}\dfrac{f_1^2(r(q))}{\widetilde{g}\left(\widetilde{\nabla}r(q),\widetilde{\nabla}r(q)\right)}\,d{\rm A}_{\widetilde{g}}(q)\right)dt\\
		& = \int_{{\rm B}_R(p)}\phi_1^2(q)\,d{\rm V}_{\widetilde{g}}(q).
	\end{align*}
	Similarly, since the nominator $g(\nabla \phi_1(q),\nabla \phi_1(q))=(f_1'(r(q))^2$ is also a radial function,
	$$
	\int_{{\rm B}_R(p)}g( \nabla \phi_1(q), \nabla \phi_1(q))\,d{\rm V}_g(q)=\int_{{\rm B}_R(p)}\widetilde{g}(\widetilde{\nabla} \phi_1(q),\widetilde{\nabla} \phi_1(q))\,d{\rm V}_{\widetilde{g}}(q)
	$$
	Therefore, by equations \eqref{eq:Rayquo1} and \eqref{eq:Rayquo3}, by inequality \eqref{eq:Rayquo2} and by Corollary \ref{corollarifinal}, we obtain
	$$
	\lambda_{1,g}({\rm B}_R(p))\leq \lambda_{1,\widetilde{g}}({\rm B}_R(p))={\lim_{k\to \infty}\left(\frac{\int_0^RT_k^2(t){\rm A}_g(t)\,dt}{\int_0^RT_{k+1}^2(t){\rm A}_g(t)\,dt}\right)^{1/2}}.
	$$
	Furthermore, equality in the above inequality implies that $\phi_1$ is also a first positive eigenfunction of $\Delta_{g}$, therefore for any $q\in {\rm B}_R(p)$
	$$
	\Delta_g\phi_1(q)=f_1''(r(q))+f'_1(r(q))\frac{\partial}{\partial r}\left(\ln{\sqrt{\det(G(r,\theta))}}(q)\right)=-\lambda_{1,g}({\rm B}_R(p))f_1(r(q)).
	$$
	Then, for any point $q\in {\rm S}_t(p)$,
	\begin{equation}\label{eq:vectorial}
		\left(\vphantom{\dfrac{1}{1}}f_1''(t)+\lambda_{1,g}({\rm B}_R(p))f_1(t)\right)\partial r(q)=f_1'(t) \vec{\rm H}_{{\rm S}_t(p)}(q).    
	\end{equation}
	But, since $\phi_1$ is a first positive eigenfunction for $\Delta_{\widetilde{g}}$, by \eqref{eq:laprotsym}, 
	$$
	f_1''(t)+\lambda_{1,g}({\rm B}_R(p))f_1(t)=-(n-1)\frac{w'_g(t)}{w_g(t)} f_1'(t).
	$$
	Hence, from \eqref{eq:vectorial} and taking into account that $f_1'(t)< 0$ (see Corollary \ref{corollarifinal}), we can obtain the mean curvature vector field $\vec{\rm H}_{{\rm S}_t(p)}(q)$ for any point $q\in {\rm S}_t(p)$ as 
	$$
	\vec{\rm H}_{{\rm S}_t(p)}(q)=-(n-1)\frac{w'_g(t)}{w_g(t)}\partial r.
	$$  
	Therefore, the mean curvature of the geodesic spheres pointed inward given by
	$$
	{\rm H}_{{\rm S}_t(p)}(q)={g(}\vec{\rm H}_{{\rm S}_t(p)}(q), -\partial r{)}=(n-1)\frac{w'_g(t)}{w_g(t)}
	$$
	is a radial function as stated.
		
	On the other hand, if the mean curvature of the geodesic spheres is a radial function, \emph{i.e.}, $\vec{\rm H}_{{\rm S}_t(p)}(q)=-h(t)\partial r${,} we are proving that $\lambda_{1,g}({\rm B}_R(p))=\lambda_{1,\widetilde{g}}({\rm B}_R(p))$. Indeed, we can prove that $\phi_1$ is a positive eigenfunction of $\Delta_g$ because 
	\begin{equation}\label{equfinal}
		\begin{aligned}
			\Delta_g\phi_1(q)=&f_1''(r(q))+f'_1(r(q))\frac{\partial}{\partial r}\left(\ln{\sqrt{\det(G(r,\theta))}}(q)\right)\\
			=&f_1''(r(q))+f'_1(r(q))h(t){.}
		\end{aligned}
	\end{equation}
	But since
	$$
	\begin{aligned}
		(n-1)\frac{\omega'_g(t)}{\omega_g(t)}=&\frac{d}{dt}\ln \omega^{n-1}(t)=\left.\frac{d}{ds}\ln {\rm A}_g(s)\right\lvert_{s=t}\\
		=&\frac{1}{A_g(t)}\int_{\mathbb{S}^{n-1}_1}\left.\frac{\partial}{\partial s}\sqrt{\det(G(s,\theta))}\right\lvert_{s=t}d\theta^1\wedge\cdots\wedge d\theta^{n-1}\\
		=&\frac{1}{A_g(t)}\int_{\mathbb{S}^{n-1}_1}\frac{\left.\frac{\partial}{\partial s}\sqrt{\det(G(s,\theta))}\right\lvert_{s=t}}{\sqrt{\det(G(t,\theta))}}\sqrt{\det(G(t,\theta))}d\theta^1\wedge\cdots\wedge d\theta^{n-1}\\
		=&\frac{1}{A_g(t)}\int_{\mathbb{S}^{n-1}_1}h(t)\sqrt{\det(G(t,\theta))}d\theta^1\wedge\cdots\wedge d\theta^{n-1}\\
		=& h(t),
	\end{aligned}
	$$
	from equation \eqref{equfinal} and \eqref{eq:laprotsym}{, we have}
	$$
	\Delta_g\phi_1(q)=f_1''(r(q))+(n-1)\frac{\omega'_g(t)}{\omega_g(t)}f'_1(r(q))=\Delta_{\widetilde{g}}\phi_1=-\lambda_{1,\widetilde{g}}({\rm B}_R(p)) \phi_1  
	$$
	Hence, $\phi_1$ is a positive eigenfunction of $\Delta_g$ and the Theorem follows.
\end{proof}

\begin{remark}
	Observe that the function $h(t)$ of the Theorem \ref{MainThm} is
	$$
	h(t)=(n-1)\frac{w'_g(t)}{w_g(t)}. 
	$$
\end{remark}

\section{Proof of Theorem \ref{ChengTypeTheo}}\label{sec:ChengTypeTheo}
In this section we will prove Theorem \ref{ChengTypeTheo} which is a particular case of the following Theorem \ref{ChengTypeTheoGen} when $W(t)=S_{\kappa}(t)$ with 
$$
S_\kappa(t):=\left\lbrace
\begin{array}{rcl}
	\frac{\sin(\sqrt{\kappa}t)}{\sqrt{\kappa}},&{\rm if }&\kappa>0,\\
	t,&{\rm if }&\kappa=0,\\
	\frac{\sinh(\sqrt{-\kappa}t)}{\sqrt{-\kappa}},&{\rm if }&\kappa<0.
\end{array}
\right.
$$
\begin{theorem}\label{ChengTypeTheoGen}
	Let $(M,g)$ be a Riemannian manifold, {and} let $p\in M$ be a point of $M$ with injectivity radius ${\rm inj}_{g}(p)$. Let ${\rm B}_R(p)$ be the geodesic ball of radius $R$ centered at $p$ with $R<{\rm inj}_{g}(p)$. Let $W:[0,R]\to \mathbb{R}$ be a non-negative smooth function such that the metric tensor
	$$
	g_W=dr\otimes dr+(W^2\circ r)\pi^*g_{\mathbb{S}^{n-1}_1}
	$$
	is smooth on ${\rm B}_R(p)$. Suppose that {$R<{\rm inj}_g(p)$ and that} for any $t<{R}$ the function
	$$
	t\mapsto\frac{{\rm A}_g(t)}{{\rm A}_{g_W}(t)}
	$$
	is a decreasing function. {Then}, the first eigenvalue $\lambda_{1,g}(({\rm B}_R(p)))$ of the Laplacian for the Dirichlet problem on the geodesic ball ${\rm B}_R(p)$ of radius $R$ centered at $p$ is bounded by 
	\begin{equation}\label{cheng-ineq3}
		\lambda_{1,g}({{\rm B}_R(p)})\leq \lambda_{1,g_W}({{\rm B}_R(p)}),
	\end{equation}
	with equality in \eqref{cheng-ineq3} if and only if, for any $t\in(0,R{)}$, the mean curvature pointed inward ${\rm H}_{{\rm S}_t(p)}$ of the geodesic sphere ${\rm S}_t(p)$ is
	$$
	{\rm H}_{{\rm S}_t(p)}=(n-1)\frac{W'(t)}{W(t)}, \quad {\rm for\;{all} }\quad 0<t<R.
	$$
\end{theorem} 
\begin{proof}
	From $({\rm B}_R(p),g)$ we will symmetrize the metric tensor to obtain the rotationally symmetric metric tensor
	\begin{equation}
		\widetilde{g}=dr\otimes dr+(\omega^2_g\circ r)\pi^*g_{\mathbb{S}^{n-1}_1}
	\end{equation}
	with $\omega_g:[0,R)\to\mathbb{R}_+$ being the positive function given by 
	\begin{equation*}
		t\mapsto\omega_g(t):=\left(\frac{{\rm A}_g(t)}{{\rm vol}\left({\mathbb{S}_1^{n-1}}\right)} \right)^{\frac{1}{n-1}},
	\end{equation*}
	Hence, by using the Theorem \ref{MainThm}
	\begin{equation}\label{equnua}
		\lambda_{1,g}(({\rm B}_R(p)))\leq \lambda_{1,\widetilde{g}}(({\rm B}_R(p))).
	\end{equation}
	Moreover, since ${\rm A}_{\widetilde{g}}(t)={\rm A}_{g}(t)$, by hypothesis the function
	$$
	t\mapsto \frac{{\rm A}_{\widetilde{g}}(t)}{{\rm A}_{g_W}(t)}
	$$ 
	is assumed to be a decreasing function. Therefore
	$$
	\frac{\omega_g'(t)}{\omega_g(t)}\leq \frac{W'(t)}{W(t)},\quad \text{ for all } t\in [0,R].
	$$
	Let us denote by $\phi_{1,W}(q)=f_{1,W}(r(q))$ the first (radial) eigenfunction of the Laplacian for the Dirichlet problem  with respect to the metric tensor $g_{W}$. Then for any $q\in {\rm S}_t(p)$
	$$
	\Delta_{g_W}\phi_{1,W}(q)=f''_{1,W}(t)+(n-1)\frac{W'(t)}{W(t)}f'_{1,W}(t)=-\lambda_{1,g_{W}}\left({\rm B}_R(t)\right)f_{1,W}(t)
	$$
	Hence, for any $q\in {\rm B}_R(p)$ with $r(q)=t$,
	\begin{align*}
		-\frac{\triangle_{\widetilde{g}}\phi_{1,W}(q)}{\phi_{1,W}(q)}  =&\frac{-f''_{1,W}(t)-(n-1)\frac{\omega_g'(t)}{\omega_g(t)}f'_{1,W}(t)}{f_{1,W}(t)}\\ \leq&\frac{-f''_{1,W}(t)-(n-1)\frac{W'(t)}{W(t)}f'_{1,W}(t)}{f_{1,W}(t)}\\
		=&\lambda_{1,g_W}({\rm B}_R(p)). 
	\end{align*}
	Finally the Theorem follows by using Barta's Lemma (see \cite{Ch1} for instance) and \eqref{equnua}.	\end{proof}

 
\end{document}